\documentclass[10pt,reqno]{amsart}

\input{xy}
\xyoption{all}

\usepackage{amsmath,amsthm}
\usepackage{amssymb}
\usepackage{euscript}

\numberwithin{equation}{subsection}

\theoremstyle{definition}
\newtheorem{Def}[subsection]{Definition}
\newtheorem{Exm}[subsection]{Example}
\newtheorem{Rem}[subsection]{Remark}

\newtheorem{Thm}[subsection]{Theorem}

\newtheorem{Prop}[subsection]{Proposition}
\newtheorem{Cor}[subsection]{Corollary}

\begin{document}
\title[]{$W$-triviality of low dimensional manifolds}

\author{Aritra C Bhattacharya}
\address{619 Red Cedar Road, Wells Hall, Michigan State University, East Lansing, MI-48824, USA.}
\email{bhatta94@msu.edu}

\author{Bikramjit Kundu}
\address{Department of Mathematics, Indian Institute of Technology, Roorkee, Haridwar Road, 247667, India.}
\email{bikramju@gmail.com, bikramjit.pd@ma.iitr.ac.in}

\author{Aniruddha C Naolekar}
\address{Stat-Math Unit, Indian Statistical Institute, 8th Mile, Mysore Road, RVCE Post, Bangalore 560059, INDIA.}
\email{ani@isibang.ac.in}

\thanks{}


\begin{abstract} 
A space $X$ is $W$-trivial if for every real vector bundle $\alpha$ over $X$ the total Stiefel-Whitney class $w(\alpha)$ is 1. It follows from a result of Milnor that  if $X$ is an orientable closed smooth manifold of dimension $1,2,4$ or $8$, then $X$ is not $W$-trivial. 
In this note we completely characterize $W$-trivial orientable connected  closed smooth manifolds in dimensions $3,5$ and $6$. In dimension $7$, we describe necessary conditions for an orientable connected closed smooth $7$-manifold to be $W$-trivial. 
\end{abstract}
 
\maketitle

\section{Introduction.}

A theorem of Milnor \cite{milnor} (see also \cite[Theorem\,1, page 223]{atiyah}) states the following. 

\begin{Thm}\cite[Theorem\,1, page 444]{milnor}\label{milnor}
There exists a real vector bundle $\alpha$ over the sphere $S^n$ with $w_n(\alpha)\neq 0$ if and only if $n= 1,2,4,8$.\qed
\end{Thm}

 For $n\in\{1,2,4,8\}$, the bundles in Milnor's theorem are the famous Hopf bundles. Since $H^i(S^n;\mathbb{Z}_2)=0$ for $i\neq0,n$, for a real vector bundle $\alpha$ over $S^n$, $w(\alpha)=1$ if and only if $w_n(\alpha)=0$.

We recall the following definition from \cite{tanaka}. 

\begin{Def}
A space $X$ is said to be $W$-trivial if for every real vector bundle $\alpha$ over $X$, the total Stiefel-Whitney class $w(\alpha)$ is $1$.
\end{Def}

In view of the above definition, Milnor's theorem can be restated as: $S^n$ is $W$-trivial if and only if $n\neq 1,2,4,8$.


In this note we address the question if there are orientable connected  closed smooth manifolds other than the spheres that are $W$-trivial. We completely characterize $W$-trivial manifolds in dimensions $3,5,6$ and prove a necessary condition for a manifold to be $W$-trivial in dimension $7$. Before stating the main results we make a few remarks. 

\begin{Rem} \label{rem} 
\begin{enumerate}
\item If $n=1,2,4,8$ and $X$ is an orientable connected  closed smooth $n$-manifold, then $X$ is not $W$-trivial. This is because if $f:X\longrightarrow S^n$ is a degree $1$ map, then 
$$w_n(f^*\alpha)\neq 0$$
where $\alpha$ is the Hopf bundle on $S^n$.

\item  Recall that, for a commutative ring $R$, a $R$-homology $n$-sphere is an orientable connected  closed smooth $n$-manifold $X$ with $H_i(X;R)\cong H_i(S^n;R)$ for all $i$. It is known \cite[page 94]{milnor} that for a real vector bundle $\alpha$ the smallest positive integer $n$ for which $w_n(\alpha)\neq 0$ is a power of $2$. Thus if $X$ is a $\mathbb Z_2$-homology $n$-sphere and $n$ is not a power of $2$, then $X$ is $W$-trivial.

\item It would be interesting to know if Theorem\,\ref{milnor} is true for $\mathbb Z_2$-homology spheres. 
\end{enumerate}
\end{Rem}

We now state our main results. We assume that all manifolds that appear in the sequel are connected. 

\begin{Thm}\label{main} 
Let $X$ be an orientable closed smooth manifold of dimension either $3$ or $5$. Then $X$ is $W$-trivial if and only if  $X$ is a $\mathbb Z_2$-homology sphere. \qed
\end{Thm}

In dimension $6$ there exist $W$-trivial manifolds $X$  that are not $\mathbb Z_2$-homology spheres. For example 
$S^3\times S^3$ is $W$-trivial (see Remark\,\ref{rem} (2) above). 
In dimension $6$ we prove the following. 

\begin{Thm}\label{second}
Let $X$ be an orientable closed smooth manifold of dimension $6$. Then $X$ is $W$-trivial if and only if 
$$
H_i(X;\mathbb Z)=\left\{ \begin{array}{cl}
\mathbb Z & i=0\\
F& i=1\\
F' & i=2\\
\mathbb Z^t\oplus F' & i= 3\\
 F& i=4\\
0 & i=5\\
\mathbb Z & i=6
\end{array}\right.
$$
where $F,F'$ are finite abelian groups neither of which contain any element of order $2$ and $t\geq 0$. \qed
\end{Thm}

For example $S^3\times S^3$ is a $6$-manifold whose homology groups are of the form described in the above theorem and hence is $W$-trivial. In dimension $7$ we derive the following homological restriction on $W$-trivial manifolds. 

\begin{Thm}\label{seven}
Let $X$ be an orientable closed smooth $7$-manifold. If $X$ is $W$-trivial, then the integral homology groups of $X$ are of the form 
$$
H_i(X;\mathbb Z)=\left\{ \begin{array}{cl}
\mathbb Z & i=0\\
F& i=1\\
F' & i=2\\
F'' & i= 3\\
 F'& i=4\\
F & i=5\\
0 & i=6\\
\mathbb Z & i=7
\end{array}\right.
$$
where $F,F', F''$ are finite abelian groups with $F,F''$ not having any element of order $2$. The converse is true if in addition $F'$ does not contain any element of order $2$. \qed
\end{Thm}

In particular, if $X$ is a $W$-trivial orientable closed smooth $7$-manifold, then $X$ is a rational homology $7$-sphere. We have not been able to decide whether or not there is a $W$-trivial oriented closed smooth $7$-manifold $X$ such that $ H_2(X;\mathbb Z)$ contains an  element of order $2$. 

We end this section by  looking at some examples and non-examples. 
 Note that in dimensions $3,5$ and $7$ every $W$-trivial manifold turns out to be a rational homology sphere. 
We begin by discussing the $W$-triviality of certain familiar rational homology spheres.

\begin{Exm} Since $H^1(\mathbb RP^n;\mathbb Z_2)\neq 0$, the real projective spaces $\mathbb RP^n$ are not $W$-trivial (see Proposition\,\ref{firstprop} below).  
Let 
$$L_{n,p}=S^{2n-1}/\mathbb Z_p$$ 
$n>1$ and $p$ a prime, denote the Lens space obtained as a quotient of $S^{2n-1}$ by the diagonal action of $\mathbb Z_p$. Here we regard $\mathbb Z_p$ as the group of $p$-th roots of unity and $S^{2n-1}$ as the unit sphere in $\mathbb C^n$. 
The integral homology groups of $L_{n,p}$ are zero  in even degrees $i$, $1< i<2n-1$ and are equal to $\mathbb Z_p$ in odd degrees $i$, $1\leq i <2n-1$. Hence if $p$ odd, then the Lens space $L_{n,p}$ is $W$-trivial. 
\end{Exm}

\begin{Exm}
Let $V_2(\mathbb R^5)$ denote the Stiefel manifold of orthonormal $2$-frames in $\mathbb R^5$. $V_2(\mathbb R^5)$ is a rational homology $7$-sphere. The integral homology of $V_2(\mathbb R^5)$ is of the form as in Theorem\,\ref{seven} with $F=F'=0$ and $F'' = \mathbb Z_2$ and hence 
$V_2(\mathbb R^5)$ is not $W$-trivial. Another way to prove this is to consider the sphere bundle 
$$S^3 \hookrightarrow V_2(\mathbb R^5)\stackrel{p}\longrightarrow V_1(\mathbb R^5)=S^4.$$
A part of the Gysin sequence associated to the above sphere bundle is 
$$\rightarrow H^0(S^4;\mathbb Z_2)\longrightarrow H^4(S^4;\mathbb Z_2)\stackrel{p^*}\longrightarrow H^4(V_2(\mathbb R^5);\mathbb Z_2)\longrightarrow H^1(S^4;\mathbb Z_2) \rightarrow.$$
It is well known \cite{borel} that $H^4(V_2(\mathbb R^5);\mathbb Z_2)\cong \mathbb Z_2$. Hence $p^*$ is an isomorphism. Thus if $\alpha$ is the Hopf bundle over 
$S^4$, then $w_4(p^*(\alpha))= p^*(w_4(\alpha))\neq 0$ and hence $V_2(\mathbb R^5)$ is not $W$-trivial.

\end{Exm}


\begin{Exm}
 Examples of $W$-trivial manifolds in dimension $7$ are given by certain rational homology spheres constructed by Ruberman \cite[Example\,7, page 232]{ruberman}. In each dimension $n>5$ Ruberman has constructed examples $X_k$ of  simply connected rational homology  $n$-spheres with 
$$
H_i(X_k;\mathbb Z) =\left\{
\begin{array}{cl}
\mathbb Z & i=0,n\\
\mathbb Z_k & i=2,n-3\\
0 & \mbox{otherwise}.
\end{array}\right.
$$
In dimension $7$, these rational homology spheres $X_k$ have the homology as in Theorem\,\ref{seven} with $F=F''=0$, $F'=\mathbb Z_k$, and hence are $W$-trivial if $k$ is odd.
For a general $n> 5$, when $k$ is odd, it is clear that $X_k$ is a $\mathbb Z_2$-homology sphere. If in addition $n$ is not a power of $2$, then  $X_k$  is $W$-trivial. In each dimension $n\geq 5$ there exist, by a construction of Kervaire \cite[Theorem\,1, page 67]{kervaire}, a non simply connected integral homology sphere $K$. If the Ruberman manifold $X_k$ is $W$-trivial, then the connected sum $X_k\# K$ is $W$-trivial. This gives us examples of non simply connected $W$-trivial manifolds. 
\end{Exm}

\begin{Exm}
In dimension $5$, there exist simply connected manifolds $B_k$  (see \cite{barden}) whose integral homology is of the form 
$$
H_i(B_k;\mathbb Z)=\left\{\begin{array}{ccl}
\mathbb Z & i=0,5\\
\mathbb Z_k\oplus \mathbb Z_k &i=2\\
0 & \mbox{otherwise}.
\end{array}\right.
$$
If $k$ is odd, then $B_k$ is $W$-trivial. One may now take connected sum with the integral homology spheres of Kervaire to obtain a non  simply connected $W$-trivial $5$-manifold. 
\end{Exm}

This note is organized as follows. In section 2 we collect some preliminary facts about $W$-trivial manifolds and then prove the main theorems. Throughout, we follow the conventions stated below.

{\em Conventions.} In this note we follow the conventions of \cite{ani} which we recall briefly.  The notations  $F,F',F'',\ldots $ will denote finite abelian groups. Given $F$ (resp. $F',F''\ldots$), the integer $s$ (resp. $s',s'',\ldots$) will denote the number of primes $p_i$ that are equal to $2$ in a direct sum decomposition 
$$F=\oplus_i\mathbb Z_{p_i^{k_i}}$$
of $F$ (resp. $F',F'',\ldots$) with $p_i$ not necessarily distinct. Thus we have 
$$\mathrm{Ext}(F,\mathbb Z)=F$$
and 
$$\mathrm{Ext}(F,\mathbb Z_2)=\mathbb Z_2^s=\mathrm{Hom}(F,\mathbb Z_2).$$
Given $s$ (resp. $s',s'',\ldots$) the integers $u,v$ (resp. $u',v'; u'',v'';\ldots$ ) will denote the number of primes $p_i$ such that $k_i$ is equal to $1$ and the number of primes $p_i$ such that $k_i>1$ respectively. Thus $s=u+v$. In particular, 
$$\mathrm{Ext}(F,\mathbb Z_4)= \mathbb Z_2^u\oplus \mathbb Z_4^v.$$
Finally, all vector bundles that we consider are real unless stated otherwise.


\section{Proofs}


In this section we prove the main theorems. We begin by fixing some notations and proving some preliminary results that we shall use. 
We shall denote by $\rho_k$ the mod-$k$ reduction homomorphism 
$$\rho_k: H^i(X;\mathbb Z)\longrightarrow H^i(X;\mathbb Z_k)$$
and by  
$$\rho_{4,2}:H^4(X;\mathbb Z_4)\longrightarrow H^4(X;\mathbb Z_2)$$
the homomorphism induced by the coefficient surjection.

We shall denote by $\mathfrak{P}$
the Pontryagin square (see \cite{thomas1})
$$\mathfrak{P} : H^{2k}(X;\mathbb Z_2)\longrightarrow H^{4k}(X;\mathbb Z_4).$$
The Pontryagin square  satisfies 
$$\rho_{4,2}\circ\mathfrak{P}(x)=x^2$$
for every $x\in H^{2k}(X;\mathbb Z_2)$. 

We begin with some observations that we shall use in the sequel. Recall that given a finite abelian group $F=\oplus_i\mathbb Z_{p_i^{k_i}}$, the integer $s$ denotes the number of primes $p_i$ that are equal to $2$. 

\begin{Rem}
Let $F$ be a finite abelian group and 
$$\beta:\mathbb Z_2^{\ell}\oplus \mathbb Z_2^{s} \longrightarrow F$$
be a monomorphism. Then $\ell =0$. 
This is because the number of elements of order $2$ in $ \mathbb Z_2^{\ell}\oplus \mathbb Z_2^s$ and $F$ are $2^{\ell+s}-1$ and $2^s-1$ respectively. 
Since $2^{\ell+s}-1\leq 2^s-1$, it follows that $\ell=0$.
\end{Rem}

\begin{Prop}\label{firstprop} Let $X$ be a closed  smooth $n$-manifold. 
If $X$ is $W$-trivial, then 
\begin{enumerate}
\item $X$ is orientable,
\item $X$ is unoriented nullbordant, 
\item $H^1(X;\mathbb Z_2)=0$, 
\item $H_1(X;\mathbb Z)$ is finite and does not have elements of order $2$, 
\item $\rho_2:H^2(X;\mathbb Z)\longrightarrow H^2(X;\mathbb Z_2)$ is the zero homomorphism, 
\item $H_2(X;\mathbb Z)$ is finite, and 
\item $Sq^k:H^{n-k}(X;\mathbb Z_2)\longrightarrow H^n(X;\mathbb Z_2)$ is the zero homomorphism for all $k\geq 1$. 
\end{enumerate}
\end{Prop}
\begin{proof}
If $TX$ is the tangent bundle of $X$, then as $w_1(TX)=0$, we must have that $X$ is orientable proving (1). 
To prove (2) we note that as $w(TX)=1$ all Stiefel-Whitney numbers of $X$ must be zero. Hence by a deep theorem of Thom 
 \cite[Theorem 4.10, page 53]{milnor} $X$ is unoriented nullbordant.
Since for every $x\in H^1(X;\mathbb Z_2)$ there exists a line bundle $\alpha$ with $w_1(\alpha)=x$ we must have that $H^1(X;\mathbb Z_2)=0$ proving (3). That $H_1(X;\mathbb Z)$ is finite and does not contain elements of order two follows from (3). Next, if 
$$\rho_2:H^2(X;\mathbb Z)\longrightarrow H^2(X;\mathbb Z_2)$$
is non-zero we choose $x\in H^2(X;\mathbb Z)$, $x\neq 0$ such that $\rho_2(x)\neq 0$. If $\alpha$ is a complex line bundle with with $c_1(\alpha)=x$, then 
$$w_2(\alpha_{\mathbb R})=\rho_2(c_1(\alpha))=\rho_2(x)\neq 0.$$
Here $\alpha_{\mathbb R}$ denotes the underlying real bundle of $\alpha$. This contradiction proves (5). To prove (6), we assume that the dimension of $X$ is not equal to $1,2,4,8$ (see Remark\,\ref{rem} above). If dimension of $X$ is $3$, then $H_2(X;\mathbb Z)=0$ and we are done. So assuming that dimension of $X$ is greater than $4$, we have 
$$
H_i(X;\mathbb Z)=\left\{\begin{array}{ccl}
\mathbb Z & i=0\\
F & i=1\\
\mathbb Z^{\ell}\oplus F' & i=2\\
\mathbb Z^t\oplus F'' & i=3
\end{array}\right.
$$
where $F$ has no elements of order $2$. Since $\rho_2:H^2(X;\mathbb Z)\longrightarrow H^2(X;\mathbb Z_2)$ is the zero homomorphism 
and $F$ has no elements of order $2$, we have that 
$$\beta : H^2(X;\mathbb Z_2)\cong \mathbb Z_2^{\ell} \oplus \mathbb Z_2^{s'}\longrightarrow H^3(X;\mathbb Z)\cong \mathbb Z^t\oplus F'$$
is a monomorphism where $\beta$ is the Bockstein associated to 
$$0\longrightarrow \mathbb Z\stackrel{\times 2}\longrightarrow \mathbb Z\longrightarrow \mathbb Z_2\longrightarrow 0.$$
This forces $\ell=0$  and completes the proof of (6). Finally, as $X$ is $W$-trivial the Wu classes $v_k$, $k\geq 1$,  of $X$ are all zero. As 
$$Sq^k x= x\smile v_k=0$$
for all $x\in H^{n-k}(X;\mathbb Z_2)$ we have that (7) is true. This completes the proof of the proposition. 
\end{proof}

\begin{Prop}\label{secondprop}
Let $X$ be an orientable closed  smooth manifold. Then $X$ is $W$-trivial if any one of the following conditions is satisfied.  
\begin{enumerate}
\item $\widetilde{KO}(X)=0$,
\item $X$ is a $\mathbb Z_2$-homology $n$-sphere and $n$ is not a power of $2$.
\end{enumerate}
\end{Prop}
\begin{proof}
Recall that $ \widetilde{KO}(X)$ is the abelian group of stable isomorphism classes of real vector bundles over $X$. Two vector bundles $\alpha, \beta$ over $X$ are stably equivalent if there exist trivial bundles $\epsilon, \epsilon'$ (possibly of different ranks) such that 
$$\alpha \oplus \epsilon \cong \beta\oplus \epsilon'.$$
By the Whitney product theorem we have that $w_i(\alpha)=w_i(\alpha\oplus \epsilon)$ for all $i\geq 0$. Thus stably isomorphic vector bundles have the same Stiefel-Whitney classes. This completes the proof of (1) as $\widetilde{KO}(X)=0$ implies that every vector bundle on $X$ is stably trivial. 
The proof of (2) has been outlined in Remark\,\ref{rem} (2). This completes the proof of the proposition. 
\end{proof}

\begin{Rem} We mention two remarks related to statements in the Proposition\,\ref{firstprop}-\ref{secondprop}.  
 \begin{enumerate}
\item  As noted after Theorem\,\ref{second}, $S^3\times S^3$ is $W$-trivial. This also follows from 
Proposition\,\ref{secondprop} since $\widetilde{KO}(S^3\times S^3)=0$. 
\item The converse of Proposition\,\ref{firstprop} (2) is not true. Examples are provided by the odd dimensional real projective spaces. 
\end{enumerate}
\end{Rem}

In dimension $6$, a $W$-trivial manifold can be a product as we have seen. We note that this is not possible in dimensions $3,5$ and $7$.

\begin{Prop}
Let $X$ be a non-empty $W$-trivial closed smooth manifold of dimension $3,5,$ or $7$. Then $X$ is not a product of manifolds of positive dimension. \qed
\end{Prop}
\begin{proof}
Suppose that $X$ splits as a product of positive dimensional manifolds
$$X=X_1\times X_2$$
with $X_i$  closed and smooth for $i=1,2$. As $X$ is $W$-trivial so is $X_i$ for $i=1,2$. This is because the projection 
$$\pi_i:X_1\times X_2\longrightarrow X_i$$
induces a monomorphism $\pi_i^*:H^* (X_i;\mathbb Z_2)\longrightarrow H^*(X_1\times X_2;\mathbb Z_2)$ for $i=1,2$. 

As the dimension of $X$ is $3,5$ or $7$ it follows that, in each case, one of $X_1$ or $X_2$ must have dimension equal to $2^j$ where $j=0,1,2$. 
But by Remark\,\ref{rem} (1) no manifold of dimension $1,2$ or $4$ is $W$-trivial. This contradiction proves the proposition. 
\end{proof}


 We mention that if $X,Y$ are $\mathbb Z_2$-homology spheres of dimensions $m,n$ respectively with none of $m,n,m+n$ a power of $2$, then the product $X\times Y$ is $W$-trivial. This is because the cohomology of $X\times Y$ is concentrated in degrees $0,m,n$ and $m+n$. Since none of $m,n,m+n$ are a power $2$, by Remark\,\ref{rem} (2), all the Stiefel-Whitney classes of $X\times Y$ in positive degrees are zero. 
 
 In general the product of two $W$-trivial manifolds is not $W$-trivial. For example, the spheres $S^3,S^5$ are $W$-trivial but the product $S^3\times S^5$ is not $W$-trivial 
 by Remark\,\ref{rem} (1). We however have the following observation about $W$-triviality of products. 

\begin{Prop}
Let $X,Y$ be two orientable closed smooth manifolds with $\widetilde{KO}(X)=0$ and $\widetilde{KO}(Y)=0$. If $X\wedge Y$ is $W$-trivial, then the product $X\times Y$ is $W$-trivial. 
\end{Prop}
\begin{proof}
We shall check that $w(\alpha)=1$ for every vector bundle $\alpha$ over $X\times Y$. 
The cofibre sequence 
$$X \vee Y \longrightarrow X\times Y\stackrel{p}\longrightarrow X \wedge Y$$
gives the exact sequence 
$$\widetilde{KO}(X\wedge Y)\stackrel{p^*}\longrightarrow \widetilde{KO}(X\times Y)\longrightarrow \widetilde{KO}(X\vee Y).$$
The last group is zero and hence $p^*$ is surjective. Given a vector bundle $\alpha$ over 
$X\times Y$, there exists a vector bundle $\beta$ over $X\wedge Y$ with $p^*(\alpha)\cong \beta$. Since $Y\wedge Y$ is $W$-trivial, we have $w(\beta)=1$. This forces 
$w(\alpha)=w(p^*(\beta))=1$. This completes the proof. 
\end{proof}

\begin{Cor}
Let $X$ be an orientable closed smooth manifold with $\widetilde{KO}(X)=0$. Let $n>8$ be an integer such that $n\equiv 3,5,6,7 \pmod{8}$. Then 
$S^n\times X$ is $W$-trivial.
\end{Cor}
\begin{proof}
It is known that $\widetilde{KO}(S^n)=0$ for $n\equiv 3,5,6,7 \pmod{8}$ \cite{hus}. Since $n>8$, by \cite[Theorem\,1, page 223]{atiyah}, the wedge $S^n\wedge X= \Sigma^nX$ is $W$-trivial. Hence by the above proposition $S^n\times X$ is $W$-trivial. This completes the proof. 
\end{proof}

We now turn to the proof of the main theorems. 
We begin by proving Theorem\,\ref{main} in dimension $3$.

\begin{Prop}\label{abcd}
Let $X$ be an orientable closed smooth $3$-manifold. Then $X$ is $W$-trivial if and only if it is a $\mathbb Z_2$-homology $3$-sphere.
\end{Prop}
\begin{proof}
We first assume that $X$ is $W$-trivial. By Proposition\,\ref{firstprop} (4) 
$$H_1(X;\mathbb Z)=F$$
is finite and does not contain any element of order $2$. It follows that $H^1(X;\mathbb Z)\cong H_2(X;\mathbb Z)=0$. Hence $X$ is a $\mathbb Z_2$-homology $3$-sphere. 
The converse follows from Proposition\,\ref{secondprop} (2). This completes the proof. 
\end{proof}

To prove Theorem\,\ref{main} in dimension $5$ we shall make use of a result due to \v{C}adek-Van\v{z}ura which we state below for easy reference. 

\begin{Thm} \cite[Theorem\,1, page 757]{cadek}\label{cadek} Let $X$ be a $CW$-complex of dimension $\leq 5$ and suppose 
$$\gamma : [X,BSO(5)]\longrightarrow H^2(X;\mathbb Z_2)\oplus H^4(X;\mathbb Z_2)\oplus H^4(X;\mathbb Z)$$
is defined by 
$$\gamma(\alpha)=(w_2(\alpha),w_4(\alpha),p_1(\alpha)).$$
Then $\mathrm{image}(\gamma)=\{(a,b,c)\,:\, \rho_4(c)=\mathfrak{P}(a)+i_*(b)\}$. 
Here $i_*:H^*(X;\mathbb Z_2)\longrightarrow H^*(X;\mathbb Z_4)$ is induced by 
the coefficient injection $\mathbb Z_2\longrightarrow \mathbb Z_4$. \qed
\end{Thm}

\begin{Prop}\label{wellwell}
Let $X$  be an orientable closed  smooth $5$-manifold.  Then $X$ is $W$-trivial if and only if it is a $\mathbb Z_2$-homology $5$-sphere. 
\end{Prop}
\begin{proof} Assume that $X$ is $W$-trivial. We shall show that $X$ is a $\mathbb Z_2$-homology $5$-sphere. 
We first observe, using Proposition\,\ref{firstprop}, that the integral homology groups of $X$ are of the form 
$$
H_i(X;\mathbb Z)=\left\{ \begin{array}{cl}
\mathbb Z & i=0\\
F& i=1\\
 F' & i=2\\
\mathbb Z^t\oplus F'' & i= 3\\
\mathbb Z^{\ell'} & i=4\\
\mathbb Z & i=5
\end{array}\right.$$
where we recall that $F,F',F''$  are finite abelian groups with $F$ not containing any element of order $2$. 
Since $X$ is orientable we have, by Poincar\'{e} duality, that 
$$t=0,\,\,\,\,\ell'=0,\,\,\,\, F=F''.$$
Hence $F''$ does not contain any element of order $2$. 
Thus the mod-$2$ cohomology groups of $X$ are of the form
$$
H^i(X;\mathbb Z_2)=\left\{ \begin{array}{cl}
\mathbb Z_2 & i=0\\
0 & i=1\\
 \mathbb Z_2^{s'} & i=2\\
 \mathbb Z_2^{s'}& i= 3\\
0 & i=4\\
\mathbb Z _2& i=5.
\end{array}\right.$$
The proof that $X$ is a $\mathbb Z_2$-homology $5$-sphere will be complete if we can show that $s'=0$. 
Notice that as $\ell'=0$ and $F''=F$ do not contain any element of order $2$ we have that 
$$H^4(X;\mathbb Z_4)=0.$$
This implies that every tuple $(a,b,c)$ with $a\in H^2(X;\mathbb Z_2), b\in H^4(X;\mathbb Z_2), c\in H^4(X;\mathbb Z)$ is in the image of $\gamma$. Thus if $s'\neq 0$, then every 
non zero cohomology class $a\in H^2(X;\mathbb Z_2)$ equals $w_2(\alpha$) for some vector bundle $\alpha$ over $X$. This contradiction shows that  $s'=0$.  This completes the proof that $X$ is a $\mathbb Z_2$-homology $5$-sphere.

The converse follows from Proposition\,\ref{secondprop} (2). This together with Proposition\,\ref{abcd} completes the proof of Theorem\,\ref{main}. 
\end{proof}

We now turn to the proof of Theorem\,\ref{second}. We shall need the following theorem. 

\begin{Thm}\cite[Proposition\,2, page 729]{cadek1}\label{cadek1} Let $X$ be a connected $CW$-complex of dimension $\leq 7$ and suppose 
$$\gamma : [X,BSO(7)]\longrightarrow  H^2(X;\mathbb Z_2)\oplus H^4(X;\mathbb Z_2)\oplus H^4(X;\mathbb Z)$$
is defined by 
$$\gamma(\alpha)= (w_2(\alpha),w_4(\alpha), p_1(\alpha)).$$
Then $\mathrm{image}(\gamma)=\{(a,b,c)\,:\,\rho_4(c)=\mathfrak{P}(a)+i_*(b)\}$. Here $i_*:H^*(X;\mathbb Z_2)\longrightarrow H^*(X;\mathbb Z_4)$ is induced by 
the coefficient injection $\mathbb Z_2\longrightarrow \mathbb Z_4$. \qed
\end{Thm}

{\em Proof of Theorem\,\ref{second}}. 
We assume that $X$ is a $W$-trivial orientable closed smooth $6$-manifold. Then we shall show that the integral homology groups of $X$ are of the required form. To begin with, the integral homology groups of $X$ are of the form 
$$
H_i(X;\mathbb Z)=\left\{ \begin{array}{cl}
\mathbb Z & i=0\\
F& i=1\\
F' & i=2\\
\mathbb Z^t\oplus F'' & i= 3\\
\mathbb Z^{\ell'} \oplus F'''& i=4\\
\mathbb Z^{\ell''} & i=5\\
\mathbb Z & i=6
\end{array}\right.
$$
where we recall that $F,F',F''$ and $F'''$  are finite abelian groups with $F$ not containing any element of order $2$. 
Since $X$ is orientable, we have 
$$\ell''=0,\,\,\,\,F=F''',\,\,\,\,\ell'=0,\,\,\,F'=F''.$$
Thus $F'''$ does not contain any element of order $2$. 
The proof will be complete if we can show that $s'=0$. We assume that $s'>0$. This leads to a contradiction as follows. First observe that 
$$H^4(X;\mathbb Z_4)\cong \mathbb Z_2^{u'}\oplus\mathbb Z_4^{v'}.$$
Identifying the cohomology groups in the long exact sequence 
$$\longrightarrow H^4(X;\mathbb Z)\longrightarrow H^4(X;\mathbb Z)\stackrel{\rho_4}\longrightarrow H^4(X;\mathbb Z_4)\stackrel{\beta}\longrightarrow H^5(X;\mathbb Z)$$
associated to the short exact sequence of coefficient groups 
$$0\longrightarrow \mathbb Z\stackrel{\times 4}\longrightarrow \mathbb Z\longrightarrow \mathbb Z_4\longrightarrow 0$$
we get the exact sequence 
$$\cdots \longrightarrow F'\longrightarrow F'\stackrel{\rho_4}\longrightarrow \mathbb Z_2^{u'}\oplus\mathbb Z_4^{v'}\stackrel{\beta}\longrightarrow F\longrightarrow\cdots.$$
Since $F$ has no elements of order $2$, we have that $F$ has no elements of order $4$. This implies that the homomorphism $\beta$ is the zero homomorphism and hence 
$$\rho_4:H^4(X;\mathbb Z)\longrightarrow H^4(X;\mathbb Z_4)$$
is an epimorphism.  Next observe that as $s'\neq 0$ we have 
$$H^2(X;\mathbb Z_2)\cong \mathbb Z_2^{s'}\neq 0.$$
Let $a\in H^2(X;\mathbb Z_2)$ with $a\neq 0$. Then $\mathfrak{P}(a) \in H^4(X;\mathbb Z_4)$. As $\rho_4$ is an epimorphism, we choose $c\in H^4(X;\mathbb Z)$ 
such that 
$$\rho_4(c)=\mathfrak{P}(a).$$
Then, the $3$-tuple $(a,0,c)$ is in the image of $\gamma$ where $\gamma$ is the function defined in Theorem\,\ref{cadek1}. Thus  
there exists a vector bundle $\alpha$ over $X$ with $w_2(\alpha)=a \neq 0$. This contradiction implies that $s'=0$ and completes the proof that the 
integral homology of $X$ has the required form. 

Conversely suppose that the integral homology of $X$ has the form in the theorem. Then the mod-$2$ cohomology groups satisfy 
$H^i(X;\mathbb Z_2)=0$ for $i=1,2,4$. This clearly forces $X$ to be $W$-trivial. This completes the proof of Theorem\,\ref{second}. \qed

Next, we turn our attention to $7$-manifolds. We also recall the following theorem which we shall use. 

\begin{Thm}\cite[Theorem\,3, page 733]{cadek1} \label{cad2} Let $X$ be a $CW$ complex of dimension $\leq 7$ and let $P\in H^4(X;\mathbb Z)$, 
$W\in H^4(X;\mathbb Z_2)$. Then there exists an oriented rank $5$ bundle $\alpha$ over $X$ with 
$$w_2(\alpha)=0,\,\,\,\,\,w_4(\alpha)=W,\,\,\,\,\,p_1(\alpha)=P$$
if and only if there is $U\in H^4(X;\mathbb Z)$ such that 
\begin{enumerate}
\item $P=2U$, $\rho_2(U)=W$,
\item $Sq^2(W)=0$, and 
\item $0\in \Phi(U)$
\end{enumerate}
where $\Phi$ is the secondary cohomology operation from $H^4(X;\mathbb Z)$ into $H^7(X;\mathbb Z_2)$ associated to the relation $Sq^2\circ Sq^2\rho_2=0$. \qed
\end{Thm}

Before proving Theorem\,\ref{seven}, we make a few remarks on the secondary cohomology operation $\Phi$. Given a space $Y$, $\Phi$ is defined on 
the kernel of the homomorphism 
$$Sq^2\circ \rho_2:H^4(Y;\mathbb Z)\longrightarrow H^6(Y;\mathbb Z_2)$$ and 
takes values in $H^7(Y;\mathbb Z_2)/Q$ where $Q=Sq^2 H^5(Y;\mathbb Z_2)$. If $u\in H^4(Y;\mathbb Z)$ is in the kernel of 
$Sq^2\circ \rho_2$, then $\Phi(u)$ is defined to be the coset 
$$\Phi(u) = U^* (x) + Sq^2(H^5(Y;\mathbb Z_2)$$
where $U$ is the lift of 
$$Sq^2\circ\rho_2\circ u: X\longrightarrow K(\mathbb Z_2,6)$$
to $E$, where $E$ is the pull back of the path fibration over $K(\mathbb Z_2,6)$ to $K(\mathbb Z,4)$ via $Sq^2\circ \rho_2$. Here $x\in H^7(E;\mathbb Z_2)\cong \mathbb Z_2$ is a generator. Note that if $H^6(Y;\mathbb Z_2)=0$, then 
$\Phi(u)$ is defined for all $u\in H^4(X;\mathbb Z)$. 

{\em Proof of Theorem\,\ref{seven}}. 
Let $X$ be an orientable closed smooth $7$-manifold. Assuming that $X$ is $W$-trivial and using Proposition\,\ref{firstprop}, the integral homology groups of $X$ can be readily seen to be of the form  

$$
H_i(X;\mathbb Z)=\left\{ \begin{array}{cl}
\mathbb Z & i=0\\
F& i=1\\
 F' & i=2\\
\mathbb Z^t\oplus F'' & i= 3\\
\mathbb Z^t \oplus F'& i=4\\
 F& i=5\\
0 & i=6\\
\mathbb Z & i=7
\end{array}\right.$$
where $F,F',F''$ are finite abelian groups with $F$ not containing any element of order $2$. In particular, we note that 
$H^6(X;\mathbb Z_2)=0$
and hence the secondary operation $\Phi$  is defined on the whole of $H^4(X;\mathbb Z)$. 

To complete the proof of the theorem it is enough to prove that
$$t=s''=0.$$
To prove this 
we first claim that 
$$\rho_2:H^4(X;\mathbb Z)\longrightarrow H^4(X;\mathbb Z_2)$$
is the zero homomorphism. Assuming the contrary, suppose that there exists $U\in H^4(X;\mathbb Z)$, $U\neq 0$ with $\rho_2(U)=W\neq 0$. This leads to a contradiction as follows. 
 If we let $P=2U$ we see that the first condition of Theorem\,\ref{cad2} is satisfied. Since $H^6(X;\mathbb Z_2)=0$, it follows that 
$Sq^2(W)=0$. We may now check that $0\in \Phi(U)$ as in \cite[ page 206]{ani1}. Thus  by Theorem\,\ref{cad2} there exists a 
vector bundle $\alpha$ over $X$ with 
$$w_4(\alpha)=W\neq 0.$$
This contradiction proves the claim. We may now look at the exact sequence 
$$H^4(X;\mathbb Z)\stackrel{\rho_2}\longrightarrow H^4(X;\mathbb Z_2)\stackrel{\beta}\longrightarrow H^5(X;\mathbb Z)\longrightarrow.$$
Then as $\rho_2=0$ we have that 
$$\beta : \mathbb Z_2^t\oplus \mathbb Z_2^{s'}\oplus \mathbb Z_2^{s''}\longrightarrow  F'$$
is a monomorphism. This implies, as before, that
$$t+s''=0$$
and hence $t=s''=0$.  



Conversely, if the homology of $X$ is of the given form and if additionally, $F'$ contains no elements of order $2$, then $X$ is a $\mathbb Z_2$-homology sphere and hence $W$-trivial by Proposition\,\ref{secondprop} (2). This completes the proof.

We have not been able to decide whether or not $F'$ can contain any element of order $2$. We derive the following necessary conditions for a closed smooth $7$-manifold to be $W$-trivial. 

\begin{Cor} If $X$ is a $W$-trivial closed smooth $7$-manifold, then 
\begin{enumerate}
\item The map 
$$Sq^2: H^2(X;\mathbb Z_2)\longrightarrow H^4(X;\mathbb Z_2)$$
given by $Sq^2(x) = x\smile x$ is an isomorphism. 
\item The Pontryagin square
$$\mathfrak{P}:H^2(X;\mathbb Z_2)=\mathbb Z_2^{s'}\longrightarrow H^4(X;\mathbb Z_4)=\mathbb Z_2^{u'}\oplus \mathbb Z_4^{v'}$$
is a monomorphism. 
\end{enumerate}
\end{Cor}
\begin{proof}
By Theorem\,\ref{seven} 
$$H^2(X;\mathbb Z_2)\cong \mathbb Z_2^{s'}\cong H^4(X;\mathbb Z_2). $$
Now if $f (x)=x^2=0$ for some non-zero $x\in H^2(X;\mathbb Z_2)$, then by \cite[Theorem 1, page 176]{woodward} there exists a vector bundle $\alpha$ with $w_2(\alpha)=x\neq 0$. This contradiction proves (1). 

Next, let 
$$\rho_{4,2}:H^4(X;\mathbb Z_4)\longrightarrow H^4(X;\mathbb Z_2)$$
be the homomorphism induced by the coefficient surjection. Then as 
$$\rho_{4,2}\circ\mathfrak{P}=Sq^2,$$
(2) follows from (1). This completes the proof. 
\end{proof}



\begin{Prop}
 If $X,Y$ are closed smooth manifolds of dimension  $n=3,5$ or $6$ and are $W$-trivial, then so is their connected sum.
\end{Prop}
\begin{proof}
First note that if $X,Y$ are two orientable closed smooth $n$-manifolds, then their connected sum 
$X\# Y$ is also an orientable closed smooth manifold and 
$$H_i(X\# Y;R)\cong H_i(X;R)\oplus H_i(Y;R)$$
for $i\neq n, n-1$ and all commutative rings $R$. 

Assume that $X,Y$ are $3$-dimensional. Then by Theorem\,\ref{main}, $X$ and $Y$ 
are $\mathbb Z_2$-homology spheres. Since 
$$H_1(X\#Y;\mathbb Z_2)\cong H_1(X;\mathbb Z_2)\oplus H_1(Y;\mathbb Z_2)=0$$
it follows that $H^1(X\#Y;\mathbb Z_2)=0$. By Poincar\'{e} duality we have that 
$H_2(X\#Y;\mathbb Z_2)=0$. Thus $X\#Y$ is a $\mathbb Z_2$-homology $3$-sphere. 
Hence by Theorem\,\ref{main}, $X\#Y$ is $W$-trivial. 

The proof in the case that $X,Y$ are $5$-dimensional is similar and we omit the proof. 
Finally we assume that $X,Y$ are $6$-dimensional. The integral homology groups of $X$ and $Y$ are as given in Theorem\,\ref{second}. It can now be verified that 
the integral homology groups of the connected sum $X\#Y$ is also of the form 
as given in Theorem\,\ref{second} and hence $X\#Y$ is $W$-trivial. 
\end{proof}

{\em Acknowledgements.} We express our appreciation to the anonymous referee for her/his detailed comments and suggestions which 
has improved the presentation and corrected several inaccuracies. In particular, we would like to thank the referee for pointing out the reference  
\cite{thomas1} and suggesting a way to shorten the proof of Proposition\,\ref{wellwell}.

\section{Conflict of interest statement}

On behalf of all authors, the corresponding author states that there is no conflict of interest.
\section{Data availability statement}
We do not analyse or generate any datasets, because our work proceeds within a theoretical and mathematical approach.

\end{document}